\documentclass[12pt, reqno, a4paper]{amsart}
\usepackage{amsmath,mathrsfs}
\usepackage{amssymb}
\usepackage{color}
\usepackage{calligra}
\usepackage{dsfont}
\usepackage{pstricks,pst-node}
\usepackage[latin1]{inputenc}
\usepackage{marginnote}
\usepackage{pdfsync}
\usepackage{todonotes}
\usepackage{multibib}
\usepackage{url}

\newcommand\NoBlackBoxes{\global\overfullrule0pt}
\NoBlackBoxes

\newcommand{\N}{\mathbb{N}}

\makeatletter
\let\serieslogo@\relax
\let\@setcopyright\relax

\newtheorem{definition}{Definition}[section]
\newtheorem{theorem}[definition]{Theorem}

\newtheorem{rem}[definition]{Remark}

\newcommand{\E}{{\mathbb{E}}}

\newcommand{\R}{{\mathbb{R}}}

\newcommand{\V}{\mathbb{V}}

\newcommand{\tr}{\operatorname{Tr}}

\renewcommand{\epsilon}{\varepsilon}
\renewcommand{\phi}{\varphi}

\numberwithin{equation}{section}

\begin{document}

\setcounter{page}{1}

\title[Recovery in block spin Ising models at criticality]{Exact Recovery in block spin Ising models at the critical line}

\author[Matthias L\"owe]{Matthias L\"owe }
\address[Matthias L\"owe]{Fachbereich Mathematik und Informatik,
Universit\"at M\"unster,
Orl\'{e}ans-Ring 10,
48149 M\"unster,
Germany}

\email[Matthias L\"owe]{maloewe@math.uni-muenster.de}

\author[Kristina Schubert]{Kristina Schubert}
\address[Kristina Schubert]{Fakult\"at f\"ur  Mathematik,
Technische Universit\"at Dortmund,
Vogelpothsweg 87,
44227 Dortmund,
Germany}

\email[Kristina Schubert]{kristina.schubert@tu-dortmund.de}


\newcommand{\wlim}{\mathop{\hbox{\rm w-lim}}}
\newcommand{\na}{{\mathbb N}}
\newcommand{\re}{{\mathbb R}}

\newcommand{\vep}{\varepsilon}

\begin{abstract}
We show how to exactly reconstruct the block structure at the critical line in the so-called Ising block model. This model was recently re-introduced by Berthet, Rigollet and Srivastava in \cite{BRS17_blockmodel}. There the authors show how to exactly reconstruct blocks away from the critical line and they give an upper and a lower bound on the number of observations one needs;  thereby they
establish a minimax optimal rate (up to constants). Our technique relies on a combination of their methods with fluctuation results obtained in \cite{LSblock1}. The latter are extended to the full critical regime. We find that the number of necessary observations depends on whether the interaction parameter between two blocks is positive or negative: In the first case, there are about $N \log N$ observations required to exactly recover the block structure, while in the latter case $\sqrt N \log N$ observations suffice.
\end{abstract}

\maketitle

\section{Introduction}

 In a recent paper Berthet, Rigollet and Srivastava rediscovered a block version of the Curie-Weiss-Ising model \cite{BRS17_blockmodel}. This model had been introduced earlier in the statistical physics literature, see e.g.~\cite{gallo_contucci_CW}, \cite{gallo_barra_contucci}, \cite{fedele_contucci}, \cite{collet_CW}. Extensions of these models are studied in \cite{KLSS19} or \cite{LSV19}. The first of these papers uses a very general interaction structure, while the latter investigates the situation in the spirit of social interaction models or statistical physics models on random graphs, see \cite{opoku}, \cite{BG92},\cite{contloewe}, and \cite{kl07}. A related version of this model has been investigated using the method of moments in \cite{werner_curie_weiss},
\cite{werner_curie_weiss2} and \cite{werner_curie_weiss_crit}.

The article by Berthet et al.~is motivated by a considerable amount of articles investigating block models in the recent past, see e.g.~\cite{AL18}, \cite{GMZZ17}, \cite{BMS13}, \cite{MNS16}, \cite{Bre15}.
The model is interesting from both a probabilistic and a statistical perspective.

To define the model, one starts by partitioning the set $\{1, \ldots, N\}$ into a set $S\subset \{1, \ldots, N\}$ with $|S|=\frac{N}{2}$ and  its complement $S^c$. To this end we need to assume that $N$ is even; the model itself can be defined and analyzed for arbitrary block sizes (see \cite{KLSS19}, where large deviations and Central Limit Theorems are proved for a general block structure), but the statistical questions then become more tricky.

The set $S$ induces a Hamiltonian on the binary hypercube $\{-1,+1\}^N, N\in~\N$ given by
\begin{equation}\label{hamil1}
H_{N,\alpha, \beta, S}(\sigma):= -\frac \beta {2N} \sum_{i \sim j} \sigma_i \sigma_j -\frac \alpha {2N} \sum_{i \not\sim j} \sigma_i \sigma_j, \quad \sigma \in \{-1,+1\}^N.
\end{equation}
Here we choose $\beta>0$ and $0\le |\alpha| \le \beta$,
and we write $i \sim j$, if either $i, j \in S$ or $i,j \in S^c$ and $i \not \sim j$, otherwise.
The Hamiltonian \eqref{hamil1} (or energy function), in turn, induces a Gibbs measure on $\{-1,+1\}^N$ given by:
\begin{equation*}
\mu_{N, \alpha, \beta} (\sigma)=\mu_{N, \alpha, \beta,S} (\sigma):= \frac{e^{-H_{N,\alpha, \beta}(\sigma)}}{\sum_{\sigma'}e^{-H_{N,\alpha, \beta}(\sigma')}}=:
\frac{e^{-H_{N,\alpha, \beta}(\sigma)}}{Z_{N, \alpha,\beta}},
\end{equation*}
where $\sigma= (\sigma_i)_{i=1}^N \in \{-1,+1\}^N.$

The behavior of the model can sometimes be studied best, when analyzing an
order parameter. In this case such an order parameter is given by the vector of block magnetizations, $m:=m^N:=(m^N_1, m^N_2)$, where
$$
m_1:=m^N_1:=m_1(\sigma):= \frac {2} {  N} \sum_{i \in S} \sigma_i \quad \mbox{and } \quad  m_2:=m^N_2:=m_2(\sigma):= \frac {2} {N} \sum_{i \notin S} \sigma_i.
$$
Its advantage is that the Hamiltonian can be rewritten as
$$
H_{N,\alpha, \beta, S}(\sigma)= -\frac N 2 \left(\frac{1}{2}\alpha m_1 m_2 +\beta \frac{1}{4} m_1^2+ \frac{1}{4} \beta m_2^2\right).
$$
To describe and understand the phase transitions in the model recall that without a partitioning (and with $\alpha=\beta$) we would be back in  the situation of the so-called Curie-Weiss model at inverse temperature $\beta$. This model is defined by the Hamiltonian
$H_{CW}(\sigma)= \frac 1{2N}\sum_{i,j} \sigma_i \sigma_j$ for $\sigma \in  \{\pm 1\}^N$ and the corresponding Gibbs measure
$\mu^{CW}_{N,\beta}(\sigma)=\frac{e^{-\beta H_{CW}(\sigma)}}{Z^{CW}_{N,\beta}}$.
It has been extensively studied, see e.g.~\cite{Ellis-EntropyLargeDeviationsAndStatisticalMechanics}. In particular, it has been shown that its equilibrium measures are intrinsically related to the largest solution, $m^+(\beta)$, of the equation
\begin{equation*}
z = \tanh(\beta z).
\end{equation*}
More precisely, for all $\beta \ge 0$, in the Curie-Weiss model the random variable $\frac 1N \sum_{i=1}^N \sigma_i$ asymptotically concentrates in the points $m^+(\beta)$ and $-m^+(\beta)$. Note that $m^+(\beta)=0$, if and only if $\beta \le 1$.

A similar change in the behavior was shown to be true for the block spin Ising model, see \cite{BRS17_blockmodel}.
\begin{theorem}(cf.~\cite[Proposition 4.1]{BRS17_blockmodel})\label{theo1}
In the above setting assume that $|\alpha| \le \beta$ and denote by $\rho_{N,\alpha,\beta}$ the distribution of $m$ under the Gibbs measure $\mu_{N, \alpha, \beta} $. 
\begin{itemize}
\item If $\beta +|\alpha| \le 2$, then $\rho_{N,\alpha,\beta}$ weakly converges to the Dirac measure in $(0,0)$.
\item If $\beta +|\alpha| > 2$ and $\alpha=0$, then $\rho_{N,\alpha,\beta}$ weakly converges to the mixture of Dirac measures $\frac 1 4 \sum_{ s_1, s_2 \in \{-,+\}} \delta_{(s_1 m^+(\beta/2), s_2 m^+(\beta/2))}$.
\item If $\beta +\alpha > 2$ and $\alpha>0$, then $\rho_{N,\alpha,\beta}$ weakly converges to the mixture of Dirac measures:
$\frac 1 2 (\delta_{(m^+(\frac{\alpha+\beta}2),  m^+(\frac{\alpha+\beta}2))}+\delta_{(-m^+(\frac{\alpha+\beta}2),  -m^+(\frac{\alpha+\beta}2)}) $.
\item If $\beta +|\alpha| > 2$ and $\alpha<0$, then $\rho_{N,\alpha,\beta}$ weakly converges to the mixture of Dirac measures:
$\frac 1 2 (\delta_{(m^+(\frac{\beta-\alpha}2),  -m^+(\frac{\beta-\alpha}2))}+\delta_{(-m^+(\frac{\beta-\alpha}2),  m^+(\frac{\beta-\alpha}2)}) $.
\end{itemize}
\end{theorem}
Theorem \ref{theo1} can be considered as a  Law of Large Numbers for $m$. Central Limit Theorems for $m$ were proved in \cite{LSblock1}.

They are also very useful for understanding the following reconstruction result, even though the authors in \cite{BRS17_blockmodel} choose a different approach. We will come back to this point when having described the reconstruction mechanism.

In the major part of their work \cite{BRS17_blockmodel} Berthet et al.~consider the question, whether with a given number of observations $n$ one can reconstruct $S$ exactly, and, if so, how $n$ relates to $N$. One of their main findings is
\begin{theorem}\label{gen_recov} (cf.~\cite[Corollary 4.6]{BRS17_blockmodel})
If the parameters $\alpha$ and $\beta$ satisfy $|\alpha| \le \beta$, $\alpha < \beta$, and $|\alpha|+\beta \neq 2$, then there exist positive constants $C_1$ and $C_2$ that depend
on $\alpha$ and $\beta$ such that there is an algorithm that
recovers the block structure $(S, S^c)$ exactly with probability $1-\delta$ whenever
\begin{enumerate}
\item $n \ge C_1 N \log(N / \delta)$ if $\alpha>0$ or  $|\alpha|+\beta < 2$ or
\item $n \ge C_2 \log(N/ \delta)$, otherwise,
\end{enumerate}
where $n$ denotes the number of observations.
\end{theorem}
There are two regimes of parameters excluded by Theorem \ref{gen_recov}. The first is $\alpha = \beta$. While one can still prove limit theorems in this case (see \cite{werner_curie_weiss}), it is rather obvious that it is impossible to reconstruct $S$ in this case: The interaction simply does not differentiate between spins in the same block and spins in different blocks.

Another obvious question left open by Theorem \ref{gen_recov} is, what happens at the critical line $|\alpha|+\beta=2$. The purpose of this note is to fill this gap. We will show:

\begin{theorem}\label{Crit_recov}
If the parameters $\alpha$ and $\beta$ satisfy $\alpha <\beta$, $\alpha \neq 0$, and $|\alpha|+\beta = 2$, then
the following holds true.
\begin{itemize}
\item If $\alpha >0$, there exists a positive constant $C_3$ such that there is an algorithm that
recovers the block structure $(S, S^c)$ exactly with probability $1-\delta$ whenever
$$
n \ge C_3 N \log(N / \delta).
$$
\item If $\alpha <0$, there exists a positive constant $C_4$ such that there is an algorithm that
recovers the block structure $(S, S^c)$ exactly with probability $1-\delta$ whenever
$$
n \ge C_4 \sqrt N \log(N / \delta).
$$
\end{itemize}
Here, $n$ denotes the number of observations.
\end{theorem}
\begin{rem}
Note that the number $n$ of necessary observations indicates that the phase one enters at the phase transition point $\alpha >0$ and $\alpha+\beta=2$ is of a different nature than the phase one enters at $\alpha <0$ and $|\alpha|+\beta=2$ -- even though both points still belong to the high temperature regime according to Theorem \ref{theo1}. We exclude the case $\alpha=0$, because it is already covered by Theorem \ref{gen_recov}.
\end{rem}
\begin{rem}
Following the proof in \cite{BRS17_blockmodel} one can see that the constants $C_1$ and $C_3$ in the above Theorems \ref{gen_recov} and \ref{Crit_recov} depend on $\alpha$ and $\beta$ and explode when $\alpha$ tends to $\beta$. We will not elaborate on this point.
\end{rem}

The rest of this note is devoted to the proof of Theorem \ref{Crit_recov}. To this end we will quickly recap the general reconstruction approach from \cite{BRS17_blockmodel} in Section 2. Section 3 will generalize a result on the critical fluctuations of $m_1+m_2$ and $m_1-m_2$ from \cite{LSblock1}. In particular, we will treat the case of negative $\alpha$ which was omitted in \cite{LSblock1}. These two ingredients will yield the proof of Theorem \ref{Crit_recov}, which will be given in Section 4.

\section{The strategy for block recovery}
While part 2 of Theorem \ref{gen_recov} could, in principle, be shown using a large deviations estimate, the first part needs some more sophisticated arguments.
In this section we will recall how the approach proposed in \cite{BRS17_blockmodel} by Berthet et al.~works. We will be a bit brief and especially refer the reader to \cite{BRS17_blockmodel} for proofs.

Given $n$ observations $\sigma^{(1)}, \ldots, \sigma^{(n)}$ the log-likelihood function for $S$ is given by
$$
\mathcal{L}(S)= - n \log Z_{N,\alpha, \beta}(S)-\sum_{k=1}^n H_{N,\alpha, \beta, S}(\sigma^{(k)}).
$$
Since it is easily seen that $Z_{N, \alpha,\beta}$, as a sum over all configurations $\sigma$, is independent of $S$, because we know its size, maximizing
$\mathcal{L}(S)$ amounts to minimizing $\sum_{k=1}^n H_{N,\alpha, \beta, S}(\sigma^{(k)})$. On the other hand, taking the $N \times N$ matrix $Q$ with elements $Q_{ij}= \frac{\beta}N$, if $i \sim j$ and $Q_{ij}= \frac{\alpha}N$, otherwise , one readily sees that $H_{N,\alpha, \beta, S}(\sigma)=-\frac 12 \mathrm{Tr}(\sigma \sigma^T Q)$ (where the upper index $T$ indicates transposition). Note that $Q$ depends on $S$.

Thus finding the maximum likelihood estimator for $S$ amounts to
maximizing $\frac 12 \mathrm{Tr}(\hat \Sigma Q)$ or $\mathrm{Tr}(\hat \Sigma Q)$, where $\hat \Sigma=\frac 1 n \sum_{k=1}^n \sigma^{(k)} {\sigma^{(k)}}^T$ is the empirical covariance matrix of the observations.

Here, we maximize over all matrices $Q$ that induce a  bisection of $\{1, \ldots N\}$ into sets $S$ and $S^c$ of equal size. More precisely, the $N\times N$ matrix $Q=(Q_{ij})_{1 \leq i,j \leq N}$ has to satisfy the following properties: $Q$ is symmetric, $Q_{ij}\in \{\frac{\beta}{N}, \frac{\alpha}{N}\}$, the diagonal entries are equal to $\frac{\beta}{N}$, $Q_{ij}=Q_{jk}=\frac{\beta}{N}$ implies $Q_{ik}=\frac{\beta}{N}$ for all $i,j,k \in \{1, \ldots N\}$ and for each $i=1, \ldots, N$ we have $|\{j: Q_{ij}=\frac{\beta}{N}\}|=\frac{N}{2}$.

Now, for each fixed $\sigma$ the function $\mathrm{Tr}(\sigma {\sigma}^T Q)$ is maximized by the same set $S$ (associated to the matrix $Q$) no matter, what $\alpha$ and $\beta$ are, as long as $\alpha < \beta$. The same holds true for $n$ fixed observations $\sigma^{(1)}, \ldots, \sigma^{(n)}$ and $\mathrm{Tr}(\hat \Sigma Q)$.

Indeed, it is an easy matter to check that
$$
\frac {2n} {N}\mathrm{Tr}(\hat \Sigma Q)=\sum_{k=1}^n  \alpha m_1 (\sigma^{(k)}) m_2(\sigma^{(k)}) +\frac \beta 2 ((m_1(\sigma^{(k)}))^2+ (m_2(\sigma^{(k)}))^2).
$$
As $\sigma^{(1)}, \ldots, \sigma^{(n)}$ are fixed so is $\frac 2 N \sum_{i=1}^N \sigma_i^{(k)}=:c_k$. This means we have that $$m_2(\sigma^{(k)})= c_k - m_1(\sigma^{(k)})$$ for all $k$ and constants $c_k$ depending on $k$ (of course), but not depending on the partitioning $(S,S^c)$. Thus
given our $n$ observations $\sigma^{(1)}, \ldots, \sigma^{(n)}$ we compute
\begin{align*}
&\frac {2n} {N}\mathrm{Tr}(\hat \Sigma Q)
\\
& \quad = \alpha \sum_{k=1}^n (m_1(\sigma^{(k)}))(c_k-m_1(\sigma^{(k)}))+\frac{\beta}{2}\sum_{k=1}^n (m_1(\sigma^{(k)}))^2+(c_k-m_1(\sigma^{(k)}))^2).
\end{align*}
At first glance the right hand side may appear rather tricky, because even though the observations $\sigma^{(1)}, \ldots, \sigma^{(n)}$ are taken independently, the $m_1(\sigma^{(k)})$ are not, when $S$ is the free variable. However, multiplying out the products (resp.~the square) and rearranging the terms we see that
$$
\frac {2n} {N}\mathrm{Tr}(\hat \Sigma Q)=K+ (\beta-\alpha)\sum_{k=1}^n (m_1(\sigma^{(k)}))^2- (\beta-\alpha)\sum_{k=1}^n c_k m_1(\sigma^{(k)})
$$
for some constant $K$ depending on $\beta$ and $\sigma^{(1)}, \ldots, \sigma^{(n)}$ but not on the choice of $S$.
This reveals that the minimal and maximal points of $\frac {2n} {N}\mathrm{Tr}(\hat \Sigma Q)$ as a function of $S$ do not depend on $\alpha$ and $\beta$ as long as $\beta \neq \alpha$. Only, whether they are maxima or minima depends on whether $\alpha < \beta$ or $\alpha >\beta$.
Hence, as long as we keep $\alpha <  \beta$ we can choose any values for them we like.

We can therefore also set $\beta=N$ and $\alpha=-N$. This transforms our optimization problem into
\begin{equation}\label{max1}
\max_{R \in \mathcal{R}} \mathrm{Tr}(\hat \Sigma R) \qquad \mbox{with } \mathcal{R}=\{R= r r^T: r \in \{\pm 1\}^N: \sum_{i=1}^N r_i=0 \}
\end{equation}
(cf.~(3.2) in \cite{BRS17_blockmodel}).
Obviously each $R=rr^T \in \mathcal{R}$ again induces a bisection of $\{1, \ldots, N\}$ into two sets $S$ and $S^c$ of equal size, where $i\sim j$, if $r_i=r_j$. Moreover,  there is a one-to-one correspondence between the set of valid matrices $Q$ and $\mathcal{R}$; e.g.~given some  matrix $Q$ with the aforementioned properties, the corresponding matrix $R=rr^T \in \mathcal R$ is given by $r_i =1$, if $Q_{1i}=\frac{\beta}{N}$ and $r_i=-1$ otherwise. Thus each $R \in \mathcal{R}$ is an estimator for the unknown blocks.

In \cite{BRS17_blockmodel} the authors now proceed in two steps. First the $\sigma$ are centered, i.e.~$\sigma \in \{\pm 1\}^N$ is replaced by $\overline \sigma:= \Pi \sigma$ with $\Pi:= \mathrm{Id}_N- \frac 1 N \mathds{1}_N$, where $\mathds{1}_N$ is the
$N \times N$ matrix with all elements equal to $1$. Since for all $R \in \mathcal{R}$, we have that $\tr[\hat \Sigma R ] = \tr [\hat \Gamma R]$, where $\hat \Gamma= \Pi \hat \Sigma \Pi$ the likelihood function remains the same
over $\mathcal{R}$ when we replace $\hat \Sigma$ by $\hat \Gamma.$
The decisive step in \cite{BRS17_blockmodel} is then to embed the optimization problem \eqref{max1} into a larger class of optimization problems. Hence, instead of solving \eqref{max1} the authors look for solutions of
\begin{equation}\label{max2}
\max_{R \in \mathcal{E}^+} \mathrm{Tr}(\hat \Gamma R)
\end{equation}
with
\begin{align*}
 \mathcal{E}^+:=&\{ R: R  \text{ is a positive semidefinite, symmetric } N \times  N \text{ matrix}
\\
& \quad \text{ with } 1 \text{ on the diagonal}\} \nonumber.
\end{align*}
The question is, of course, when a solution of \eqref{max2} also provides a solution to \eqref{max1}. This is, where the authors in \cite{BRS17_blockmodel} spend a considerable amount of work to show that:

\begin{theorem}\label{helptheo}(cf.~\cite{BRS17_blockmodel}, Theorem 3.3)

The semidefinite programming problem \eqref{max2} has a unique maximum at $R^*\in \mathcal{E}^+$
with probability $1-\delta$ whenever
$$ n > C_{\alpha, \beta} \frac{\log(4 N/\delta)}{Z-Z'}(1+o_p(1)).
$$
Here $Z:= \E(\sigma_i \sigma_j)$, if $i \sim j$, $Z':= \E(\sigma_i \sigma_j)$, if $i \not\sim j$, and $C_{\alpha, \beta}$ is a constant depending on $\alpha$ and $\beta$.

Moreover, this unique solution of \eqref{max2} is of the form $R^*=r_S r_S^T$, where $r_S :=  \mathds{1}_S- \mathds{1}_{S^c}$ for a set $S$ with cardinality $\frac N2$. Thus $R^* \in \mathcal{R}$.
In particular, the semidefinite programming solution, if it exists, recovers exactly the block structure $(S, S^c)$.
\end{theorem}

\section{A limit theorem}
Theorem \ref{helptheo} obviously asks for an estimate of $Z-Z'$. As a matter of fact, in \cite{BRS17_blockmodel} the authors show by a comparison of the distribution of $(m_1, m_2)$ with a Gaussian distribution that for $\alpha \le 0$ and $|\alpha| + \beta >2$, the difference $Z-Z'$ is of constant order in $N$, while it is of order $\frac 1 N$, in all other cases, whenever $|\alpha|+\beta \neq 2$. Indeed, for $\alpha >0$ and $\alpha + \beta <2$ this also easily follows from the Central Limit Theorem 1.2 for the vector $m$ in \cite{LSblock1} and it would follow from corresponding Central Limit Theorems also in the other cases when $|\alpha|+\beta \neq 2$, if such were proven (they are quite likely true, because similar Central Limit Theorems hold for the vector $\frac 1N \sum_{i=1}^N \sigma_i$ in the Curie-Weiss model, see e.g.~\cite{werner_momente}).

For $|\alpha|+\beta = 2$ there is no such estimate for $Z-Z'$ in \cite{BRS17_blockmodel}. Indeed, in view of Theorem \ref{crit_fluc} below the main tool in this reference does not work, because at least for $\alpha>0$, and $\alpha+\beta=2$ the vector $\sqrt N(m_1, m_2)$ is not asymptotically Gaussian. Moreover, considering Theorem 6 and Remark 7 in \cite{werner_curie_weiss_crit} one may wonder, whether an exact reconstruction of $(S, S^c)$ on the basis of the correlations between the spins is possible at all. There, the authors show that for $\alpha>0$ and $\alpha+\beta=2$ asymptotically:
$$
\E(\sigma_i \sigma_j)= \sqrt{\frac{12}{N}} \frac{\Gamma(\frac 34)}{\Gamma(\frac 14)} \qquad \mbox{for }i\neq j
$$
{\it independent} of whether the sites $i$ and $j$ belong to the same block, or to different blocks. Hence, asymptotically the correlations $\E(\sigma_i \sigma_j)$ at $\alpha+\beta=2$ do not depend on whether $i \sim j$ or $i \not \sim j$, nor do they depend on $\alpha$ and $\beta$.

In the sequel we generalize a theorem from \cite{LSblock1} to analyze the fluctuations of $(m_1, m_2)$ for all $|\alpha| < \beta$ with $|\alpha|+\beta=2$. This will help us to prove Theorem \ref{Crit_recov}. Indeed, at
$|\alpha|+\beta= 2$ the following holds:

\begin{theorem}\label{crit_fluc}
For the parameters of the block spin Ising model $\alpha$ and $\beta$ assume that $|\alpha| < \beta$ and that $|\alpha|+\beta=2$.
\begin{enumerate}
\item[(a)]Then, if $\alpha>0$
on a scale $\sqrt N/2$ the difference between $ m_1$ and $m_2$, i.e.~$\tilde m_1 - \tilde m_2:= \frac{\sqrt N}{2} (m_1 - m_2)$, is asymptotically Gaussian with mean $0$ and variance $\frac 2 {2-(\beta-\alpha)}$.

If $\alpha<0$,
$\frac{N^{\frac 14}}{2} (m_1-m_2)$ converges in distribution to a probability measure $\rho$ on $\R$, which is absolutely continuous with Lebesgue-density
$$g(x) = \exp\left(-\frac{1}{12} x^4\right)/K,$$ where $K$ is a normalizing constant to make $\rho$ a probability measure.
\item[(b)] For the overall spin $m_1+m_2$, we have the following:  If $\alpha>0$  there is a non-standard Central Limit  Theorem, i.e.~$\frac{N^{\frac{1}{4}}}{2}(m_1+m_2)$ converges in distribution to the probability measure $\rho$ given in part (a) of the theorem. If $\alpha<0$,  there is a Gaussian Central Limit Theorem, i.e. $\frac{\sqrt{N}}{2}(m_1+m_2)$ converges in distribution to a Gaussian with mean zero and variance $-\frac{1}{\alpha}$.
\end{enumerate}
\end{theorem}

\begin{rem}
 Theorem \ref{crit_fluc} also shows the difference between $|\alpha|+\beta=2, \alpha>0$ and $|\alpha|+\beta=2, \alpha<0$. Indeed, according to Theorem \ref{theo1} for $|\alpha|+\beta$ slightly larger than 2, the block magnetizations $m_1$ and $m_2$ are close together when  $\alpha>0$, while they move apart for $\alpha<0$.

Moreover, notice that the probability measure $\rho$ in Theorem \ref{crit_fluc} is the same as the limiting distribution of the appropriately rescaled magnetization in the Curie-Weiss model, see \cite[TheoremV.9.5]{Ellis-EntropyLargeDeviationsAndStatisticalMechanics}
\end{rem}
\begin{proof}
The proof of this theorem makes use of the so-called Hubbard-Stratonovich transform.

Introduce the random variables
\begin{equation*}
w_1:= w_1(\sigma):=\frac 1 N \sum_{i=1}^N \sigma_i= \frac 1 N\left(\sum_{i \in S} \sigma_i +\sum_{i \notin S} \sigma_i\right)
\end{equation*}
and
\begin{equation*} w_2:= w_2(\sigma):= \frac 1 N\left(\sum_{i \in S} \sigma_i -\sum_{i \notin S} \sigma_i\right)
\end{equation*}
together with the vectors $\tilde w=(\tilde w_1, \tilde w_2)$ and $\hat w=(\hat w_1, \hat w_2)$ consisting of the components
$$
\tilde w_1:= \sqrt N w_1 \qquad \mbox{and } \quad \tilde w_2:= N^{1/4} w_2,
$$
as well as
$$
\hat w_1:= N^{1/4} w_1 \qquad \mbox{and } \quad \hat w_2:= \sqrt N w_2.
$$
Note that
\begin{align*}
H_{N,\alpha, \beta, S}(\sigma)&= - \frac{N}{2} \left(2 \alpha \frac14 m_1 m_2 + \beta \frac14 m_1^2 + \beta \frac14 m_2^2\right)\\
&=- \frac{N}{8} \left(2 \alpha (w_1+w_2)(w_1-w_2) + \beta\left( (w_1+w_2)^2+ (w_1-w_2)^2\right)\right)\\
&=-\frac{N}{4} (\beta ({w_1}^2 +{w_2}^2) + \alpha ({w_1}^2 - {w_2}^2))
\\
&=-\frac{1}{4} (\sqrt N(\alpha + \beta) \hat{w_1}^2 + (\beta-\alpha) \hat{w_2}^2)\\
&= -\frac{1}{4} (\sqrt N \kappa_\alpha \hat{w_1}^2 + \eta_\alpha \hat{w_2}^2),
\end{align*}
where we have set
$$
\kappa_\alpha:= \alpha+\beta = \left\{\begin{array}{ll} 2 & \mbox {if }\alpha > 0 \\ 2(1+\alpha) & \mbox {if }\alpha < 0 \end{array} \right.
$$
and
$$
\eta_\alpha:= \beta-\alpha = \left\{\begin{array}{ll} 2 & \mbox {if }\alpha < 0 \\ 2(1-\alpha) & \mbox {if }\alpha > 0 \end{array}. \right.
$$

Note that $\kappa_\alpha>0$ as well as $\eta_\alpha>0$, if $|\alpha| <\beta$ and $|\alpha|+\beta =2$.

Similarly,
$$
H_{N,\alpha, \beta, S}(\sigma)= -\frac{1}{4} (\kappa_\alpha \tilde{w_1}^2 + \sqrt N\eta_\alpha \tilde{w_2}^2).
$$

We now begin with the case of $\alpha <0$. Then $\kappa_\alpha <2$, which will make a difference, as we will see.

Our strategy can be summarized as ``linearizing'' the Hamiltonian by tilting it with a suitable Gaussian random variable (a similar technique was used e.g.~in \cite{gentzloewe}). To this end, let ${\mathcal N}(0, C)$ be a 2-dimensional Gaussian distribution with expectation $0$ and covariance matrix $C$,
$$
C=  \begin{pmatrix}
\frac{2}{\kappa_\alpha} & 0\\
0 & \frac{1}{\sqrt N}
\end{pmatrix}.
$$
Denote by $\overline \rho_{N,\alpha,\beta}:= \mu_{N,\alpha,\beta}\circ(\tilde{w})^{-1}$ the distribution of $\tilde w$ under the Gibbs measure and by
$\chi_{N,\alpha,\beta} :=\overline \rho_{N,\alpha,\beta} * {\mathcal N}(0, C)$. Let us compute the Lebesgue density of $\chi_{N,\alpha,\beta}$: Let $A$ be a measurable subset of ${\mathbb R}^2$.
Then with
$$
K_1:= \frac 1 {2 \pi \sqrt{\mathrm{det}\,C}}, \quad K_2 := \frac 1 {{2 \pi \sqrt{\mathrm{det}\,C} Z_{N,\alpha,\beta}}}
\quad \mbox{and } \quad  K_3:=K_2 2^N
$$
we obtain
\begin{eqnarray*}
& & \chi_{N,\alpha, \beta} (A) =
\sum_{\sigma \in \{-1,1\}^N} {\mathcal N}(0,C) (A-\tilde{w}) \mu_{N,\alpha,\beta}(\sigma)\\
& = & K_1
\sum_{\sigma \in \{-1,1\}^N}
\int_{A- {\tilde{w}} }    \exp \left( -\frac{1}{2}\left( \frac{\kappa_\alpha}{2}x^2  +
\sqrt N y^2 \right)\right) \mu_{N,\alpha,\beta}(\sigma) dx dy\\
& = & K_1
\sum_{\sigma \in \{-1,1\}^N}
\int_{A}    \exp \left( -\frac{1}{2}\left( \frac{\kappa_\alpha}{2}(x- \tilde{w_1})^2  +
\sqrt N (y-\tilde{w_2})^2\right)\right) \mu_{N,\alpha,\beta}(\sigma) dx dy\\
& = & K_2
\sum_{\sigma \in \{-1,1\}^N}
\int_{A}    \exp \left( -\frac{1}{2}\left( \frac{\kappa_\alpha}{2}(x- \tilde{w_1})^2  +
\sqrt N (y-\tilde{w_2})^2\right)
+\frac{1}{4} \left(\kappa_\alpha \tilde{w_1}^2 +
2\sqrt N \tilde{w_2}^2\right)\right) dx dy
\end{eqnarray*}
because $\eta_\alpha=2$ when $\alpha <0$. Thus
\begin{eqnarray}
& & \chi_{N,\alpha, \beta} (A)
=  K_2
\sum_{\sigma \in \{-1,1\}^N}
\int_{A}    \exp \left( -\frac{1}{4} \kappa_\alpha x^2 -
\frac{\sqrt N}2 y^2
+\frac{1}{2} \kappa_\alpha x \tilde{w_1} +
\sqrt N y \tilde{w_2}\right)dx dy \nonumber \\
& = & K_3
\int_{A}    \exp\left( -\frac{1}{4} \kappa_\alpha x^2 -
\frac{\sqrt N} 2 y^2 \right)\exp\left( \frac N2 \log \cosh \left( \frac{1}{\sqrt{N}}\frac{\kappa_\alpha}{2} x +
N^{-\frac 14} y\right)\right.\nonumber  \\
& & \left. \qquad \qquad +  \frac N2 \log \cosh \left( \frac{\kappa_\alpha}{2\sqrt N} x -
\frac{1}{N^{\frac 14}} y\right) \right)dx dy. \label{density_1}
\end{eqnarray}
Note that in the above expression, both the integral term and the constant $K_3$ depend on $N$. However, it suffices to consider the convergence of the integral, because once the convergence of the  integral is shown for all measurable subsets   $A \subset \mathbb R^2$ and in particular for $A= \mathbb R^2$, this implies the convergence of $K_3$ to a constant.

Introduce as short-hand for the negative exponent in \eqref{density_1}:
\begin{equation}\label{def_Phi}
\Phi(x,y):= \frac{1}{4} \kappa_\alpha x^2 +
\frac{\sqrt N} 2 y^2
-
\frac N2 \left(\log \cosh \left(\frac{\kappa_\alpha}{2\sqrt N} x + \frac{1}{N^{\frac 14}} y\right)
+\log \cosh \left(\frac{\kappa_\alpha}{2\sqrt N} x - \frac{1}{N^{\frac 14}} y\right)\right).
\end{equation}
Then by Taylor expansion of
$\log \cosh (z) = \frac{1}{2}z^2-\frac 1 {12} z^4+\mathcal{O}(z^6)$ up to fourth order we see that the $y^2$-terms cancel, and so do the $xy$-terms and we arrive at
\begin{equation*}
\Phi(x,y)
= \frac{x^2}{2} \left( \frac{\kappa_\alpha}{2} - \frac{\kappa^2_\alpha}{4}\right)+
\frac{y^4}{12}
+\mathcal{O}(N^{-\frac 12}),
\end{equation*}
where the constant in the $\mathcal{O}(N^{-\frac 12})$-term depends on $x$ and $y$.
Therefore we obtain for the density of the convolution
\begin{equation*}
\chi_{N,\alpha, \beta} (A)
 = K_3 \int_A\exp  \left[ -\frac{1}{2} x^2\left(  \frac{\kappa_\alpha}{2} -\frac{\kappa_\alpha^2}{4} \right)-\frac 1{12} y^4 +\mathcal{O}(N^{-\frac 12 })\right]dx \,dy\\
\end{equation*}
and the convergence in the $\mathcal{O}(N^{-\frac 12})$-term is uniform on compact subsets of $\R^2$.
From here it becomes immediately clear, that the cases $\kappa_\alpha=2$ and $\eta_\alpha=2$ have to be treated separately.

We next show how to extend the convergence to non-compact sets. This is done in the spirit of \cite{LSblock1} or \cite{gentzloewe}.

To this end for a Borel set $A \subset \R^2$ we split the integral
\begin{align}\label{split}
\int_A & \exp  \left( -\Phi(x,y)\right) dx dy =\int_{A \cap B(0,R)} \exp  \left( -\Phi(x,y)\right)dx dy\\
& + \int_{A \cap B(0,R)^c \cap B_{r,N}}\exp  \left( -\Phi(x,y)\right)dx dy
 + \int_{A \cap B_{r,N}^c}  \exp  \left( -\Phi(x,y)\right)dx dy\nonumber.
 \end{align}
Here, for any number $R >0$ we denote by $B(0,R)$ the (2-dimensional) ball centered in $0$ with radius $R$
and by $B_{r,N}$ we denote the box
\begin{equation}\label{def_B_r}
B_{r,N}:= (- r\sqrt N, r\sqrt N) \, \times (-r N^{1/4}, r N^{1/4}). 
\end{equation}

Later $R$   will be sent to $\infty$   and $r$ will be taken very small. Note that we will however first take the limit $N\to \infty$ and then $R\to \infty$. The summands in \eqref{split} will be called inner region, intermediate region, and outer region, respectively. They will be treated separately.

As noted earlier for fixed $R>0$
\begin{align*}
&\lim_{N\to \infty} \int_{A \cap B(0,R)} \exp  \left( -\Phi(x,y)\right)dx dy
\\
= & \int_{A \cap B(0,R)} \exp  \left[ -\frac{1}{2} x^2\left(  \frac{\kappa_\alpha}{2} -\frac{\kappa\alpha^2}{4} \right)
-   \frac{1}{12} y^4\right]dx \,dy.
\end{align*}
This already finishes the treatment of the inner region.

For the outer region $A \cap B_{r,N}^c$, which we will further split up in two regions, we change scales and we write
$$
\Phi(x,y) = N \tilde \Phi\left(\frac x {\sqrt N}, \frac y {N^{\frac 14}} \right),
$$
where obviously
\begin{equation*}
\tilde \Phi(x,y)
:=  \frac{1}{4} \kappa_\alpha x^2 +   \frac{1}{2}y^2
-
 \frac 12 \log \cosh \left(\frac{\kappa_\alpha}{2} x + y\right)
- \frac 12 \log \cosh \left( \frac{\kappa_\alpha}{2} x -y\right).
\end{equation*}

We denote the free energy of the Curie-Weiss model at inverse temperature $\beta$ by $f_{CW}(\beta)$. Note that
$f_{CW}(\beta)$ is non-negative and given by
$$
f_{CW}(\beta)= -\frac \beta 2 (m^+(\beta)^2+ \log (\cosh(\beta m^+(\beta)))=\max_t [-\frac{1}{2\beta}t^2+\log \cosh(t)].
$$
Then,
$$
\log\cosh(x) \le  \frac 1 4 x^2 +\max_t [-\frac{1}{4}t^2+\log \cosh(t)]=\frac 1 4 x^2+ f_{CW}(2) .
$$
Hence,
\begin{align*}
\tilde \Phi(x,y)&=\frac{1}{4} \kappa_\alpha x^2 +   \frac{1}{2}y^2
-
 \frac 12 \log \cosh \left(\frac{\kappa_\alpha}{2} x + y\right)
- \frac 12 \log \cosh \left( \frac{\kappa_\alpha}{2} x -y\right) \\
&\ge \frac{1}{4} \kappa_\alpha x^2 +   \frac{1}{2}y^2 - \frac 12 \frac{(\frac{\kappa_\alpha x}2+y)^2}4
 - \frac 12 \frac{(\frac{\kappa_\alpha x}2-y)^2}4- f_{CW}(2) \\
 &=
x^2 \left( \frac{1}{4} \kappa_\alpha-\frac{1}{16} \kappa^2_\alpha\right)+\frac{ y^2}{4}-f_{CW}(2).
\end{align*}

Setting $$r_0:= 4\sqrt{\frac{
 f_{CW}(2)}{\min(\kappa_\alpha-\frac{\kappa_\alpha^2}4,1)}},$$
and, in accordance with \eqref{def_B_r}, 
$$B_{r_0, N}:= (- r_0\sqrt N, r_0\sqrt N) \, \times (-r_0 N^{1/4}, r_0 N^{1/4}),$$
we assume that $r$ is sufficiently small such that $r<r_0$, i.e.~$B_{r,N} \subset B_{r_0,N}$, and we split up the outer region 
$$A \cap B_{r,N}^c= (A \cap B_{r_0, N}^c) \cup (A \cap (B_{r_0,N} \setminus B_{r,N})).$$

Then, if $(x,y) \notin B_{r_0, N}$ we have $\frac{x^2}{N} + \frac{y^2}{\sqrt{N}} \geq r_0^2$ and hence we know that
\begin{align*}
-\Phi(x,y)  &= -N \tilde \Phi\left(\frac x {\sqrt N}, \frac y {N^{\frac 14}} \right)
\le -N \left(\frac{x^2}N \left( \frac{1}{4} \kappa_\alpha-\frac{1}{16} \kappa^2_\alpha\right)+\frac{ y^2}{4 \sqrt N}\right)+Nf_{CW}(2)\\
&  
= -\frac{N}{16}  \left(\frac{x^2}N \left(  \kappa_\alpha-\frac{1}{4} \kappa^2_\alpha\right)+\frac{ y^2}{ \sqrt N}\right)-
\frac{3N}4 \left(\frac{x^2}N \left( \frac{1}{4} \kappa_\alpha-\frac{1}{16} \kappa^2_\alpha\right)+\frac{ y^2}{4 \sqrt N}\right)+Nf_{CW}(2)\\
&\le -\frac{3N}4 \left(\frac{x^2}N \left( \frac{1}{4} \kappa_\alpha-\frac{1}{16} \kappa^2_\alpha \right)+\frac{ y^2}{4 \sqrt N}\right)
\\ &\le - N K_\alpha \left(\frac{x^2}N+\frac{y^2}{\sqrt N}\right)
\end{align*}
for some appropriate constant $K_\alpha>0$ depending on $\alpha$.

Thus
\begin{align}
\int_{A \cap B_{r_0, N}^c}  \exp  \left( -\Phi(x,y)\right)dx dy & = \int_{A \cap B_{r_0, N}^c}  \exp   \left(-N \tilde \Phi\left(\frac x {\sqrt N}, \frac y {N^{\frac 14}} \right) \right)dx dy\nonumber\\
&\le \int_{A \cap B_{r_0, N}^c}  \exp   \left(- N K_\alpha \left(\frac{x^2}N+\frac{y^2}{\sqrt N}\right)\right)dxdy\notag\\
&\le e^{-N K_\alpha r_0^2/2} \int_{\R^2}  \exp  \left( - \frac{N K_\alpha}2 \left(\frac{x^2}N+\frac{y^2}{\sqrt N}\right)\right)dx dy \nonumber \\
&\le K'_\alpha e^{-N K_\alpha r_0^2/2} \label{farout}
\end{align}
for another constant $K_\alpha'>0$.

Next we want to extend \eqref{farout} to arbitrarily small $r \in (0, r_0)$. To this end we claim
\begin{equation}\label{claim_outer_region}
\inf_{x \in B_{r_0,N} \setminus B_{r,N}} \tilde \Phi\left(\frac x {\sqrt N}, \frac y {N^{\frac 14}} \right) >0, \quad \text{for all }0<r<r_0.
\end{equation}

By rescaling it suffices to consider $\tilde \Phi$ on the set 
$(-r_0,r_0) \times (-r_0, r_0)\setminus (-r,r)\times (-r, r )$.

Obviously, $\tilde \Phi(0,0)=0$. We want to show that $(0,0)$ is the unique minimum of $\tilde \Phi$.

Indeed, $\tilde \Phi$ gets minimal if $\nabla \tilde \Phi=0$ and the latter is given by
$$
\nabla \tilde \Phi(x,y) = \left( \begin{array}{l}
\frac 12 \kappa_\alpha x - \frac{\kappa_\alpha}4 \tanh(\frac{\kappa_\alpha x}2+y)-\frac{\kappa_\alpha}4 \tanh(\frac{\kappa_\alpha x}2-y)\\
 y - \frac{1}2 \tanh(\frac{\kappa_\alpha x}2+y)+\frac{1}2 \tanh(\frac{\kappa_\alpha x}2-y)
\end{array}
\right).
$$
Solving $\nabla \tilde \Phi=0$ thus is equivalent to solving
$$
G(x,y):= \left( \begin{array}{l}
\frac 12 \tanh(\frac{\kappa_\alpha x}2+y)+\frac{1}2 \tanh(\frac{\kappa_\alpha x}2-y)\\
\frac 12 \tanh(\frac{\kappa_\alpha x}2+y)-\frac{1}2 \tanh(\frac{\kappa_\alpha x}2-y)
\end{array}
\right)= \left(\begin{array}{l} x \\ y\end{array}\right).
$$
Obviously, $(0,0)$ is a solution.
To see that $(0,0)$ is indeed the only solution, we assume that   $G(x,y)=(x,y)$ for some $(x,y)$. Note that   $x=0$ immediately implies $y=0$. If $y=0$, we have $\tanh(\frac{\kappa_{\alpha}}{2}x)=x$, which has only the solution $x=0$, since $\frac{\kappa_{\alpha}}{2}\leq 1$.

Hence,  we can assume that $x \neq 0$ and $y \neq 0$.  By the symmetry of $\tanh$, the first equation induced by $G(x,y)=(x,y)$ reads
 \begin{equation*}
\frac{ \tanh(y+\frac{\kappa_\alpha x}2)- \tanh(y-\frac{\kappa_\alpha x}2)}{\kappa_{\alpha}x} =\frac{2}{\kappa_{\alpha}}.
\end{equation*}
With  $g(z):=(\tanh(x))'= 1- (\tanh(z))^2$ and the mean value theorem we have
\begin{equation}\label{mean_value_1}
\frac{ \tanh(y+\frac{\kappa_\alpha x}2)- \tanh(y-\frac{\kappa_\alpha x}2)}{\kappa_{\alpha}x} =g(x_0)\stackrel{!}{=} \frac{2}{\kappa_{\alpha}} = \frac{1}{1+ \alpha} >1,
\end{equation}
for some $x_0$ between $y+ \frac{\kappa_{\alpha x}}{2}$ and $y- \frac{\kappa_{\alpha x}}{2}$ (recall $\alpha<0$ for the last inequality).
Since  $0\leq g(z) \leq 1$ for all $z \in \mathbb R$ the in equality in \eqref{mean_value_1} can never be true, and there is no solution to $G(x,y)=(x,y)$ with $x \neq 0, y \neq 0$.
This, in turn implies that on $\mathbb R^2$ the function $\tilde \Phi$  has a unique minimum
at  $(0,0)$.
This proves exactly the claim in \eqref{claim_outer_region} or in other words 
 for all $r>0$ there is a constant $\xi(r, r_0)$ such that 
$$
\tilde \Phi\left( \frac{x}{\sqrt N},\frac{y}{N^{1/4}}\right) >\xi(r, r_0)
$$
for all $(x,y) \in B_{r_0,N}\setminus B_{r,N}$. 
From here we get that 
\begin{align}
\int_{A \cap (B_{r_0,N}\setminus B_{r,N})}  \exp  \left( -\Phi(x,y)\right)dx dy 
& \le \int_{B_{r_0,N}}
e^{-N \xi(r, r_0)} dxdy\notag\\
&\le 2r_0^2 N e^{-N \xi(r, r_0)}\notag\\
&\le e^{-N \xi(r, r_0)/2}\label{out}
\end{align}
for $N$ sufficiently large. 
By \eqref{farout} and \eqref{out} we have 
$$
\lim_{N \to \infty} \int_{A \cap  B_{r,N}^c}  \exp  \left( -\Phi(x,y)\right)\, dxdy=0
$$ 
for  $r>0$ sufficiently small.

Now we turn to the intermediate region $A \cap B(0,R)^c \cap B_{r,N}$: Recalling the definition of $\tilde{\Phi}$ we want to expand $   \log \cosh \left(\frac{\kappa_\alpha}{2\sqrt N} x + \frac{1}{N^{\frac 14}} y\right)
$ resp. $\log \cosh \left(\frac{\kappa_\alpha}{2\sqrt N} x - \frac{1}{N^{\frac 14}} y\right) $ for $x,y \in B_{r,N}\setminus B(0,R)$, where we recall the definition of $B_{r,N}$ given in \eqref{def_B_r}.
Note that $x,y \in B_{r,N}$ implies $\frac{\kappa_\alpha}{2\sqrt N} x \pm \frac{1}{N^{\frac 14}} y \in (-2r, 2r)$.

For $z \in (-2r, 2r)$ we use the Taylor approximation and estimate the remainder using the Lagrange form
 $$\left|\log \cosh (z)- \frac {z^2} 2 +\frac1{12} z^4 \right|\leq C r z^6 \quad \text{resp.} \quad -\log \cosh (z) \geq -\frac{z^2}{2} + \frac{1}{12}z^4 -Crz^6$$
 for some positive constant $C>0$.
Applying this to $\tilde \Phi$ and $x,y \in B_{r,N}\setminus B(0,R)$ and using that the terms with odd powers of $x$ (and $y$) cancel, this gives
\begin{eqnarray*}
&&N \tilde \Phi\left(\frac x{\sqrt N},\frac y {N^{1/4}}\right)
\geq
\frac{x^2 \kappa_\alpha}{4}+\frac{\sqrt N y^2}{2}
-\frac N4 \left(\frac{\kappa_\alpha x}{2\sqrt N} + \frac{y}{N^{1/4}}\right)^2
\\&&\quad
-\frac N4 \left(\frac{\kappa_\alpha x}{2\sqrt N} - \frac{y }{N^{1/4}}\right)^2
+ \frac N{24}  \left[\left(\frac{\kappa_\alpha x}{2\sqrt N}+\frac{y}{N^{1/4}}\right)^4
+\left(\frac{\kappa_\alpha x}{2\sqrt N}-\frac{y}{N^{1/4}}\right)^4\right]\\
&&\quad - C r \frac{N}{2}
\left[\left(\frac{\kappa_\alpha x}{2\sqrt N}+\frac{y}{N^{1/4}}\right)^6
+\left(\frac{\kappa_\alpha x}{2\sqrt N}-\frac{y}{N^{1/4}}\right)^6\right]\\
\\
&\ge &
\frac{x^2}{2} \left( \frac{\kappa_\alpha}{2} - \left( \frac{\kappa_\alpha}{2}\right)^2\right)+
\frac{y^4}{12} +\frac N2  \frac{\kappa_\alpha^2 x^2}{4 N}\frac{y^2}{\sqrt N}
\\
&&
\quad -
C r N \left[\frac{\kappa_\alpha^6 x^6}{64 N^{3}}+15\frac{\kappa_\alpha^4 x^4}{16 N^2}\frac{y^2}{\sqrt N}+
15\frac{\kappa_\alpha^2 x^2}{4 N}\frac{y^4}{N}+\frac{y^6}{N^{3/2}}\right]
\\
&\ge &
\frac{x^2}{2} \left( \frac{\kappa_\alpha}{2} - \left( \frac{\kappa_\alpha}{2}\right)^2\right)+
\frac{y^4}{12}  -
C r N \left[\frac{\kappa_\alpha^6 x^6}{64 N^{3}}+15\frac{\kappa_\alpha^4 x^4}{16 N^2}\frac{y^2}{\sqrt N}+
15\frac{\kappa_\alpha^2 x^2}{4 N}\frac{y^4}{N}+\frac{y^6}{N^{3/2}}\right].
\end{eqnarray*}
Now note that for $x,y \in B_{r,N}$ we have
\begin{equation}\label{rest2}
C r N \frac{\kappa_\alpha^6 x^6}{64 N^{3}} \le C  \frac{1}{64} \kappa_\alpha^6 x^2 r^5,
\end{equation}
\begin{equation}\label{rest3}
15 C r N \frac{\kappa_\alpha^4 x^4}{16 N^2}\frac{y^2}{\sqrt N}\le  C  \frac{15}{16} \kappa_\alpha^4 x^2 r^5,
\end{equation}
\begin{equation}\label{rest4}
15 C r N \frac{\kappa_\alpha^2 x^2}{4 N}\frac{y^4}{N}\le  C  \frac{15}{4} \kappa_\alpha^2 x^2 r^5,
\end{equation}
and
\begin{equation}\label{rest5}
C r N \frac{y^6}{N^{3/2}} \le C  y^4 r^3.
\end{equation}
Note that the coefficients of $x^2$ resp.~of $y^4$ on the right hand sides of equations \eqref{rest2}-\eqref{rest5} become arbitrarily small, if we take $r$ small enough.

On the set $A_{R,N}^r:= A \cap B(0,R)^c \cap B_{r,N}$ we can therefore bound
\begin{eqnarray*}
&&N \tilde \Phi\left(\frac x{\sqrt N},\frac y {N^{1/4}}\right)\\
&\ge&
\frac{x^2}{2} \left( \frac{\kappa_\alpha}{2} - \left( \frac{\kappa_\alpha}{2}\right)^2
-C \frac{1}{32} \kappa_\alpha^6 r^5-C \frac{15}{8} \kappa_\alpha^4 r^5-
\frac{15}{2}C \kappa_\alpha^2 r^5
\right)+
\frac{y^4}{12}(1-12 C r^3).
\end{eqnarray*}
Since $\kappa_\alpha <2$, there exists some $\varepsilon >0$ such that for $r$ sufficiently small we have
$$
 \frac{\kappa_\alpha}{2} - \left( \frac{\kappa_\alpha}{2}\right)^2
-C \frac{1}{32} \kappa_\alpha^6 r^5-C \frac{15}{8} \kappa_\alpha^4 r^5-
\frac{15}{2}C \kappa_\alpha^2 r^5
 > \varepsilon
$$
as well as
$$
1-12 C r^3 >\varepsilon.
$$

Hence for such a small value of $r$ and $N \in \mathbb N$  we get
\begin{equation}\label{intermed_est_1}
\int_{A_{R,N}^r}  \exp  \left( -\Phi(x,y)\right) dx\, dy \le  \int_{A_{R,N}^r}  \exp  \left( -\frac{\varepsilon}{2}x^2 -
\frac{\varepsilon}{12}y^4\right)dx \, dy.
\end{equation}
Note that the integrand on the left hand side of \eqref{intermed_est_1} implicitly depends on $N$ while the integrand on the right right hand side does not depend on $N$ and is integrable over $\mathbb R^2$, 
i.e.~the estimate in \eqref{intermed_est_1} carries over if we take $N \to \infty$ on both sides
\begin{equation}\label{intermed_est_2}
\lim_{N \to \infty} \int_{ A_{R,N}^r}  \exp  \left( -\Phi(x,y)\right) dx\, dy \le  \int_{A \cap B(0,R)^c}  \exp  \left( -\frac{\varepsilon}{2}x^2 -
\frac{\varepsilon}{12}y^4\right)dx \, dy.
\end{equation}
Further, as again the integrand on the right hand side of \eqref{intermed_est_2}  is integrable over all of $\R^2$,  the right hand side of \eqref{intermed_est_2} vanishes for $R \to \infty$, i.e.
\begin{align*}
\lim_{R \to \infty} \lim_{N \to\infty} \int_{A_{R,N}^r}  \exp  \left( -\Phi(x,y)\right) dx\, dy =0.
\end{align*}

Combining the three considerations above shows  that the only contribution to
$ \chi_{N,\alpha, \beta}$ comes from the inner region
\begin{equation*}
\lim_{R \to \infty}\lim_{N \to \infty}\int_A  \exp  \left( -\Phi(x,y)\right) dx dy =\int_{A } \exp  \left[ -\frac{1}{2} x^2\left(  \frac{\kappa_\alpha}{2} -\frac{\kappa\alpha^2}{4} \right)
-   \frac{1}{12} y^4\right]dx \,dy.
\end{equation*}

 Thus in parti\-cular, the distribution of $\tilde{w_2}=\frac{N^{1/4}}{2} (m_1-m_2)$ convoluted with an independent Gaussian distribution with mean $0$ and variance $\frac 1 {\sqrt N}$ converges to a probability measure $\rho$ on $\R$ with density
proportional to $\exp(-\frac 1{12} x^4)$. Since the $\mathcal{N}(0, \frac 1{\sqrt N})$-distribution converges to 0, as $N \to \infty$, this implies that $\frac{N^{1/4}}{2} (m_1-m_2)$  converges to $\rho$ in distribution, as well.
Moreover, the above calculation yields that the distribution of $\tilde{w_1}= \frac{\sqrt{N}}{2}(m_1 +m_2)$ convoluted with an independent Gaussian $\mathcal N (0, \frac{2}{\kappa_\alpha})$ converges to a Gaussian variable with mean zero and variance $\left( \frac{\kappa_{\alpha}}{2} - \frac{\kappa_{\alpha^2}}{4}\right)^{-1}$. In terms of characteristic functions this means
\begin{equation*}
\lim_{n \to \infty}\mathbb E (e^{it \tilde{w_1}})e^{-\frac{t^2}{2}\frac{2}{\kappa_{\alpha}}}= e^{-\frac{t^2}{2} \left( \frac{\kappa_{\alpha}}{2}- \frac{\kappa_{\alpha}^2}{4}\right)^{-1}}
\end{equation*}
resp.
\begin{equation*}
\lim_{n \to \infty}\mathbb E (e^{it \tilde{w_1}})= e^{-\frac{t^2}{2} \left(\frac{1}{  \frac{\kappa_{\alpha}}{2}- \frac{\kappa_{\alpha}^2}{4}} - \frac{2}{\kappa_{\alpha}}\right)}=e^{-\frac{t^2}{2} \left(\frac{1}{1-\frac{\kappa_{\alpha}}{2}}\right)}=e^{-\frac{t^2}{2} \left(- \frac{1}{\alpha}\right)}.
\end{equation*}
Hence, $\frac{\sqrt{N}}{2}(m_1+m_2)$ converges in distribution to a Gaussian distribution with mean zero and variance $-\frac{1}{\alpha}.$

To analyze the second situation, i.e.~$\alpha >0$ and $\alpha+\beta =2$, we now convolute the distribution of $\hat w$ with respect to $\mu_{N,\alpha,\beta, S}$ with a 2-dimensional normal distribution $\mathcal{N}(0, C)$, where this time
$$
C:=  \begin{pmatrix}
\frac{1}{\sqrt N} & 0\\
0 & \frac{2}{\beta-\alpha}
\end{pmatrix}=\begin{pmatrix}
\frac{1}{\sqrt N} & 0\\
0 & \frac{2}{\eta_{\alpha}}
\end{pmatrix}
$$
(recall that $\beta-\alpha >0$).
Calling this convolution $\hat \chi_{N,\alpha,\beta}$, i.e.~$\hat \chi_{N,\alpha,\beta} :=
\mu_{N,\alpha,\beta}(\hat{w})^{-1} * {\mathcal N}(0, C)$, we can compute its density as above (recall that $\kappa_{\alpha}=2$ in this case).
For a 2-dimensional Borel set $A$
we obtain
\begin{align}
&\hat \chi_{N,\alpha, \beta} (A)=\sum_{\sigma \in \{-1,1\}^N}
{\mathcal N}(0,C) (A-\hat{w})\mu_{N,\alpha,\beta}(\sigma)\nonumber \\
&=  K_2
\sum_{\sigma \in \{-1,1\}^N}
\int_{A}    \exp \left( -\frac{1}{2}\left( \sqrt{N}(x- \hat{w_1})^2  +
\frac{\eta_{\alpha}}{2} (y-\hat{w_2})^2\right)
+\frac{1}{4} \left(\sqrt{N}2 \hat{w_1}^2 +
\eta_{\alpha}  \hat{w_2}^2\right)\right) dx dy
\nonumber \\
&=  K_2
\sum_{\sigma \in \{-1,1\}^N}
\int_{A}    \exp \left(-\frac{\sqrt{N}}{2}x^2+\sqrt{N} x \hat{w_1} - \frac{\eta_{\alpha}}{4}y^2 + \frac{\eta_{\alpha}}{2} y \hat{w_2}\right) dx dy
\nonumber  \\
& =  K_3
\int_{A}    \exp\left( - \frac{ \sqrt N}{2}x^2 -   \frac{\eta_{\alpha}}{4}   y^2 \right)\exp\left( \frac N2 \log \cosh \left( \frac{ x}{N^{1/4}} + \frac{\eta_{\alpha}}{2\sqrt N} y\right)\right. \nonumber \\
&  \qquad \qquad \left.+  \frac N2 \log \cosh \left( \frac{ x}{N^{1/4}} - \frac{\eta_{\alpha}}{2\sqrt N} y\right)\right)dx dy \label{rep_density}
\end{align}
where again
$$
K_2 := \frac 1 {{2 \pi \sqrt{\mathrm{det}\,C} Z_{N,\alpha,\beta}}}, \quad \quad K_3:=2^N  \frac 1 {{2 \pi \sqrt{\mathrm{det}\,C} Z_{N,\alpha,\beta}}}.
$$
(but this time $C$ is different).

Expanding $\log \cosh$ again up to fourth order we see that now the $x^2$-terms in the exponent cancel
and we obtain
\begin{align*}
&-\frac{\sqrt N}{2}x^2 -   \frac{\eta_{\alpha}}{4} y^2 + \frac N2 \left(\log \cosh \left( \frac{x}{N^{1/4}} + \frac{\eta_{\alpha}}{2\sqrt N} y\right)+  \log \cosh \left( \frac{x}{N^{1/4}} - \frac{\eta_{\alpha}}{2\sqrt N} y\right)\right)\\
=& -\frac 1 {12} x^4 -   \frac{1}{2} y^2\left( \frac{ \eta_{\alpha}}{2} -\left( \frac{  \eta_{\alpha}}{2}\right)^2 \right)+\mathcal{O}(N^{-\frac 12})
\end{align*}
with a $\mathcal{O}(N^{-\frac 12})$-term that depends on $x$ and $y$ but is uniformly bounded for $x$ and $y$ taken from a compact subset of $\R^2$.

For non-compact sets $A$ we note that the integrand in \eqref{rep_density} equals the respective integrand in the case $\alpha<0$, if we interchange the roles of $x$ and $y$ and replace $\eta_{\alpha}$ by $\kappa_{\alpha}$ (and use $\log \cosh (z)=\log \cosh(-z)$), i.e. using the notation  $A':= \{(y,x): (x,y) \in A\}$ we have
\begin{align*}
&\hat \chi_{N,\alpha, \beta} (A) = K_3 \int_{A'}    \exp\left( - \frac{ \sqrt N}{2}y^2 -   \frac{\eta_{\alpha}}{4}   x^2 \right)\exp\left( \frac N2 \log \cosh \left( \frac{ y}{N^{1/4}} + \frac{\eta_{\alpha}}{2\sqrt N} x\right)\right. \nonumber \\
&  \qquad \qquad \left.+  \frac N2 \log \cosh \left( \frac{\eta_{\alpha}}{2\sqrt N} x- \frac{ y}{N^{1/4}}\right)\right)dx dy.
\end{align*}
Again dividing the regime of integration  into an inner, an intermediate and an outer region, we already saw in the case $\alpha <0$ (note that in the case $\alpha<0$ we did not use the explicit form of $\kappa_{\alpha}$ in that part of the proof and hence we may replace $\kappa_{\alpha}$ by $\eta_{\alpha}$)
$$
\lim_{N\to \infty} \hat \chi_{N,\alpha, \beta} (A)= \int_A \exp\left(-\frac 1 {12} x^4 -   \frac{1}{2} y^2\left( \frac{  \eta_{\alpha}}{2} -\left( \frac{  \eta_{\alpha}}{2}\right)^2 \right) \right) dx\, dy.
$$
In particular, $\hat{w_2}=\frac{\sqrt N}{2} (m_1-m_2)$ converges in distribution to a $\mathcal{N}(0, \tilde \sigma^2)$-distribution, where
$\tilde \sigma^2=\frac 1{\frac{ \eta_{\alpha}}{2} -\left( \frac{  \eta_{\alpha}}{2}\right)^2}= \frac 1{\frac{  \beta - \alpha}{2} -\left( \frac{  \beta -\alpha}{2}\right)^2}$.

This weak convergence is equivalent to the convergence of the characteristic functions.

Computing the characteristic functions of the Gaussian distribution involved in the above proof, we have therefore shown that the characteristic function of $\hat w_2$ in the point $t\in \R$ satisfies
\begin{equation*}
\lim_{N \to \infty} \mathbb E(e^{it \hat{w_2}}) e^{-\frac{1}{2}(\frac{2}{\beta- \alpha})t^2 }=e^{-\frac{1}{2} t^2  [  \frac{  \beta - \alpha}{2} -\left( \frac{  \beta -\alpha}{2}\right)^2]^{-1}}.
\end{equation*}
Therefore,
\begin{equation*}
\lim_{N \to \infty} \mathbb E(e^{it \hat{w_2}}) = e^{ \frac{1}{2}(\frac{2}{\beta- \alpha})t^2 } e^{- \frac{1}{2} t^2 [  \frac{  \beta - \alpha}{2} -\left( \frac{  \beta -\alpha}{2}\right)^2]^{-1}}
=
e^{-\frac{1}{2}t^2 \left(\frac{1}{1-\frac{ \beta - \alpha}{2}}\right)}.
\end{equation*}
On the level of weak convergence, or convergence in distribution, this in turn implies that
\begin{equation*}
\hat{w_2} \xrightarrow{N\to \infty} \mathcal N(0, \sigma^2), \quad \sigma^2= \frac{1}{1-\frac{ \beta - \alpha}{2}}
\end{equation*}
in distribution.
Similarly, $\hat{w_1}=\frac{N^{\frac{1}{4}}}{2}(m_1+m_2)$ convoluted with an independent Gaussian distribution with mean $0$ and variance $\frac 1 {\sqrt N}$ converges to a probability measure $\rho$ on $\R$ with density
proportional to $\exp(-\frac 1{12} x^4)$. Since the $\mathcal{N}(0, \frac 1{\sqrt N})$-distribution converges to 0, as $N \to \infty$, this implies that $\frac{N^{1/4}}{2} (m_1-m_2)$  converges to $\rho$ in distribution, as well.
\end{proof}

\section{Proof of Theorem \ref{Crit_recov}}

We are now ready to give the proof of our central result, Theorem \ref{Crit_recov}, which consists of a combination of the previous two sections.

\begin{proof}[Proof of Theorem \ref{Crit_recov}]
In view of Theorem \ref{helptheo}, all we need to do is to estimate
$Z-Z'$ in the critical case.

Let us first assume that $\alpha >0$.
Then, Theorem \ref{crit_fluc} shows that $\frac{1}{\sqrt N} (\sum_{i\in S}\sigma_i-\sum_{j \notin S}\sigma_j)$ converges to
a normal distribution with mean $0$ and variance $\frac{2}{2-\beta+\alpha}$. In particular
$$
\V\left(\frac{1}{\sqrt N} \left(\sum_{i\in S}\sigma_i-\sum_{j \notin S}\sigma_j\right)\right) \to \frac{2}{2-\beta+\alpha}.
$$
But,
$$
\E \frac{1}{\sqrt N} \left(\sum_{i\in S}\sigma_i-\sum_{j \notin S}\sigma_j\right) =0,
$$
hence
\begin{eqnarray*}
&&\V\left(\frac{1}{\sqrt N} \left(\sum_{i\in S}\sigma_i-\sum_{j \notin S}\sigma_j\right)\right)
\\&=&
\frac{1}N\left(\sum_{i,j\in S}\E(\sigma_i \sigma_j)+\sum_{i,j\notin S}\E(\sigma_i \sigma_j)-2\sum_{k\in S, l\notin S}\E(\sigma_k \sigma_l)\right)\\
&=& \frac 1N\left( N+2 \frac{N}{2}\left(\frac{N}{2}-1\right)Z-2\frac{N^2}4 Z'\right),
\end{eqnarray*}
where we recall that $Z:= \E(\sigma_i \sigma_j)$, if $i \sim j$ and $Z':= \E(\sigma_i \sigma_j)$, if $i \not\sim j$.
Thus
\begin{equation}\label{diff1}
\left(\frac{N}{2}-1\right)Z-\frac N2 Z' \to \frac{2}{2-\beta+\alpha}-1= \frac{\beta-\alpha}{2-\beta+\alpha}.
\end{equation}
Note that we excluded the case $\alpha =0, \beta=2$ in which case the right hand side would explode.

Note that $Z$ is of order $\frac{1}{\sqrt N}$. On the one hand this follows  from Theorem 6 and the remark from the paper by Kirsch and Toth \cite{werner_curie_weiss_crit} cited at the beginning of Section 3 and on the other hand it also follows from Theorem \ref{crit_fluc} directly: As noted in Theorem \ref{crit_fluc} (b) $\frac 1{N^{3/4}} \sum_i \sigma_i$ converges to a non-degenerate limit distribution. Since $Z \ge |Z'|$ as a consequence of $\beta \ge|\alpha|$ ($\alpha=1$, i.e.~$\alpha=\beta$ being excluded, of course) this limit theorem implies that
$\frac{N(N-1)Z}{N^{3/2}}\le c$ for some constant $c$ which shows that $c$ is at most of order $\frac 1{\sqrt N}$. But then \eqref{diff1}
shows that
\begin{equation}\label{diff2}
Z-Z' \sim C /N \qquad \mbox{for some constant } C>0
\end{equation}
where $\sim$ indicates asymptotic equivalence.

Let us now turn to the other case where $\alpha <0$ and still $\beta+|\alpha|=2$. Then according to Theorem \ref{crit_fluc} the difference of the spins $\frac{1}{N^{3/4}} (\sum_{i\in S}\sigma_i-\sum_{j \notin S}\sigma_j)$ converges to
a distribution $\rho$ on $\R$ with density proportional to $\exp(-\frac 1{12} x^4)$. In particular, $\rho$ has finite variance
$$\tau:=\sqrt{12} \frac{\Gamma(\frac 34)}{\Gamma(\frac 14)}\sim 1.17$$
Then
$$
\V\left(\frac{1}{N^{3/4}} \left(\sum_{i\in S}\sigma_i-\sum_{j \notin S}\sigma_j\right)\right) \to \tau.
$$
Taking into account that again $\E ( \sum_{i\in S}\sigma_i-\sum_{j \notin S}\sigma_j) =0$ and following the same calculations as above we obtain:
\begin{equation} \label{diff3}
\frac{1}{\sqrt N}+\left(\frac{\sqrt N}2-\frac{1}{\sqrt{N}}\right)Z-\frac{\sqrt N}2 Z' \to \tau.
\end{equation}
The observation that this time $\frac 1 {\sqrt N} \sum_i \sigma_i$ converges to a normal distribution yields that $Z$ is of order $\frac 1N$, hence \eqref{diff3} implies
\begin{equation}\label{diff4}
Z-Z' \sim \tilde C /\sqrt{N} \qquad \mbox{for some (other) constant } \tilde C>0.
\end{equation}
Combining \eqref{diff2} and \eqref{diff4}, respectively, with the reconstruction result Theorem \ref{helptheo}  from \cite{BRS17_blockmodel} implies the assertion of Theorem \ref{Crit_recov}.
\end{proof}

\medskip
\noindent{\bf Acknowledgment}: Research of the first author was funded by the Deutsche Forschungsgemeinschaft (DFG, German Research Foundation) under Germany's Excellence Strategy EXC 2044 - 390685587, Mathematics M\"unster: Dynamics-Geometry-Structure. The second author has been supported by the German Research Foundation (DFG) via Research Training Group RTG 2131 High dimensional phenomena in probability -- fluctuations and discontinuity.

The authors are thankful to the anonymous referees for their very careful reading and constructive suggestions.


\end{document}